\documentclass[12pt]{amsart}
\usepackage{amssymb,amsmath,amsfonts}
\usepackage{mathrsfs}
\usepackage{ae,aecompl}
\usepackage[T1]{fontenc}
\usepackage[pdftex]{graphicx}
\usepackage{enumerate}
\usepackage[usenames]{color}
\usepackage{epstopdf}

\DeclareMathOperator*{\g}{\gamma}
\DeclareMathOperator{\dist}{dist}

\newcommand{\PP}{\mathbb{P}}
\newcommand{\RR}{\mathbb{R}}
\newcommand{\NN}{\mathbb{N}}

\newcommand{\E}{\mathcal{E}}
\newcommand{\I}{\mathcal{I}}
\newcommand{\D}{\mathcal{D}}
\newcommand{\K}{\mathcal{K}}

\newcommand{\G}{\mathcal{G}}
\newcommand{\C}{\mathcal{C}}

\newtheorem{theorem}{Theorem}[section]
\newtheorem{lemma}[theorem]{Lemma}
\newtheorem{corollary}[theorem]{Corollary}
\newtheorem{proposition}[theorem]{Proposition}
\newtheorem{question}[theorem]{Question}
\theoremstyle{definition}
\newtheorem{definition}[theorem]{Definition}

\theoremstyle{remark}
\newtheorem{remark}[theorem]{Remark}

\numberwithin{equation}{section}

\date{\today}

\begin{document}
\title{Intersection properties of typical compact sets}
\author{Changhao Chen }


\address{Department of Mathematical Sciences, University of Oulu, P.O. Box 3000, 90014 Oulu, Finland. }\email{changhao.chen@oulu.fi}
\date{\today}



\thanks{The author acknowledge the support of the Academy of Finland, the Centre of Excellence in Analysis
and Dynamics Research.}

\begin{abstract}
We prove that a typical compact set does not contain any similar copy of a given pattern. We also prove that a typical compact set of $[0,1]^{d} (d\geq 2)$ intersects any $(d-1)$-dimensional plane in at most $d$ points. We study the ``hitting probabilities'' of compact sets in the sense of Baire category. In the end we study the arithmetic properties of typical compact sets in $[0,1]$ and the ``hitting probabilities'' of continuous functions.
\end{abstract}
\maketitle
\section{Introduction}

A subset of a metric space $X$ is of first category if it is a countable union of \emph{nowhere dense} sets (i.e. whose closure in $X$ has empty interior); otherwise it is called of second category.
We say that a \emph{typical} element $x \in X$ has property $P$, if the complement of 
\[
 \{x \in X : x \text{ satisfies } P\}
\] 
is of first category. For the basic properties and various applications of Baire Category, we refer to 
\cite{Oxtoby, Stein}. Let $\K=\K([0,1]^{d})$ be all the compact subsets of unite cube $[0,1]^{d}$. We endow $\K$ with \emph{Hausdorff metric}. Recall that the Hausdorff distance of two compact sets $E$ and $F$ of $\K$ is defined by
\[
 d_H(E,F)=\inf\{\varepsilon>0: E \subset F^\varepsilon\text{ and }F\subset E^\varepsilon\},
\]
where $E^\varepsilon=\{x\in \RR^{d}:\dist(x,E)<\varepsilon\}$.



Davies, Mastrand and Taylor \cite{Davies} constructed a compact set $A\subset [0,1]$ with Hausdorff dimension zero containing a similar copy of any finite set. Chen and Rossi \cite{ChenRossi} showed that a typical compact set is locally rich which means that we can ``see'' all the compact sets when we zooming in at any point of this compact set. Feng and Wu \cite{FengWu} proved that a typical compact set has Hausdorff dimension zero.  It is natural to ask that does a typical compact set containing a similar copy of any finite set. We have the following negative answer. 

\begin{theorem}\label{thm:type}
A typical compact set does not contain a similar copy of a given set $P$ with three distinct points.
\end{theorem}

Note that the Lebesgue density theorem implies that any set of $\RR^{d}$ with positive Lebesgue measure contains a similar copy of any finite set. However, Keleti \cite{Keleti1998, Keleti2008} constructed an $1$-dimensional compact set that does not contain the non-trivial $3$-term arithmetic progressions. Recently, Shmerkin \cite{Shmerkin} constructed an $1$-dimensional Salem set without $3$-term arithmetic progressions also. For more backgrounds  and further results we refer to \cite{Bennett, Chan, Laba,Shmerkin2}. For the basic properties of Hausdorff dimension we refer to \cite{Falconer, Mattila}.

It is not hard to see that if the complement of $A\subset \RR^{d}$ is of first category, then $A$ contains a similar copy of any countable set. This follows by the fact that  for any countable set $\{t_i \in  \RR^{d}: i\in \NN\}$, the intersection $\bigcap_{i\in\NN} (A+t_i)$ is not empty. 
Note that $A$ is not a compact set. However, there exists a second category set $E$ in the plane such that any line intersects $E$ in at most two points, see \cite[Theorem 15.5]{Oxtoby}. For a typical compact set of $\K$ we have the following result.

\begin{theorem}\label{thm:coline}
A typical compact set of $\K([0,1]^{d})(d\geq 2)$ intersects any $(d-1)$-dimensional plane in at most $d$ points.
\end{theorem}

Let $A \subset [0,1]^d$ and $\K_A=\{E\in \K: E \cap A \neq \emptyset\}.$ It is reasonable to think that if $A$ is a ``small'' set in $[0,1]^{d}$ then $\K_A$ will be a ``small'' set in $\K$ also.

\begin{theorem}\label{thm:hitting}
A set $A \subset [0,1]^d$ is nowhere dense  in  $[0,1]^d$ if and only if $\K_A$ is nowhere dense  in $\K$.
\end{theorem}

If $A\subset [0,1]^{d}$ is of first category in $[0,1]^{d}$, then $A=\bigcup_{i \in \NN}A_i$ where each $A_i$ is nowhere dense in 
$[0,1]^{d}$. Observe that 
\[
\K_A=\bigcup_{i\in \NN} \K_{A_i}.
\]
Theorem \ref{thm:hitting} claims that $\K_{A_i}$ is nowhere dense in $\K$ for each $i\in\NN$, and hence $\K_A$ is of first category in $\K$. It follows that a typical compact set of $\K$ does not intersects $A$. \v{S}al\'{a}t \cite{Salat} proved that the set of \emph{normal numbers} is of first category. It is also known that the complementary set of \emph{Liouville numbers} is of first category, see \cite[Chapter 2]{Oxtoby}. Thus we obtain that a typical compact set of $\K([0,1])$ is a subset of non-normal Liouville numbers. We collect these facts as the following corollary. 



\begin{corollary}\label{thm:normal}

$(a)$ If $A\subset [0,1]^{d}$ is of first category in $[0,1]^{d}$ then $\K_A$ is of first category in $\K$. 

$(b)$ A typical compact set of $\K([0,1])$ is a subset of non-normal Liouville numbers.
\end{corollary}

We do not know that whether $\K_A$ is of first category implies that $A$ is of first category. 

In the following, we study the size of sets formulated under finite steps arithmetic operations of a typical set $A$ of $\K([0,1])$. Let $S^{m}(A)$ be the $m$-th sum set of $A$, and  $\widetilde{P}(A)$ be a set formed under the rule of the polynomial $P$. We show these definitions in Section \ref{section:arithmetic}. Under these notations we have the following result.

\begin{theorem}\label{thm:arithmetic}
For a typical compact set $A$ of $\K([0,1])$, we have that $\dim_H S^{m}(A)=0$ for any $m\in \NN$ and $\dim_H \widetilde{P}(A)=0$ for any polynomial $P$.
\end{theorem}

The paper is organized as follows. Theorems \ref{thm:type} and \ref{thm:coline} are proved in section \ref{section:H}.  Theorem \ref{thm:hitting} is proved in section  \ref{section:hitting}. Theorem \ref{thm:arithmetic} is proved in Section \ref{section:arithmetic}. In the end we study the `` hitting probabilities'' of continuous function.

\section{Proofs of Theorems \ref{thm:type} and \ref{thm:coline}} \label{section:H}

Let $A, B \subset \RR^{d}$. If there exist $\lambda>0$ and an isometric map $\varphi$ on $\RR^{d}$ such that $\varphi(\lambda A)=B$, then we say that $A$ is similar to $B$ and denote this by $A\backsim B$. If there is a subset $A'\subset A$ such that $A'\backsim B$, then we say that $A$ contains a similar copy of $B$. For each $n\in \NN$, let $\D= \bigcup_{n\in \NN}\D_n$ where $\D_n$ is the family of  $2^{n}$-adic closed subcubes of $[0,1]^d$, i.e. 	
\[
\D_n=\Big\{\prod^d_{k=1}[i_{k}2^{-n}, (i_{k}+1)2^{-n}]:  0\leq i_{k}\leq 2^{n}-1 \Big\}.
\]

For a set $A\subset \RR^{d}$, denote by $\partial A$ the boundary of $A$, denote by $|A|$ the diameter of $A$. For two points $x, y \in \RR^{d}$, denote by  $|x-y|$ the Euclidean metric. Denote by $\mathcal{L}^{d}$ the $d$-dimensional Lebesgue measure. Let $\{x_1,x_2,x_3\} \subset \RR^{d}$  be three distinct points. Define
\[
R(x_1,x_2,x_3)=\Big\{\frac{|x_i-x_k|}{|x_j-x_k|}: i\neq j, j\neq k, i\neq k \Big\}.
\]
\begin{lemma}\label{lem:triangle}
Let $P=\{p_1,p_2,p_3\}$ be three distinct points of $\RR^{d}$ and $a,b \in\RR^{d}, a\neq b$. Then $\mathcal{L}^{d}(P')=0$, where 
\[
P'=\{x\in \RR^{d}: P\backsim \{a,b,x\} \}.
\]
\end{lemma}
\begin{proof}
If $P\backsim \{a,b,x\}$, then  $R(p_1,p_2,p_3)=R(a,b,x)$. It follows that there are at most $N=N(d)$ balls $\{B_i\}_{i=1}^{N}$ such that $P'\subset \cup_{i=1}^{N}\partial B_i$, and hence $\mathcal{L}^{d}(P')=0.$
\end{proof}

\begin{lemma}\label{lem:distribution}
Let $P=\{p_1,p_2,p_3\}$ be three distinct points of $[0,1]^{d}$. Then for any $n\in \NN$, there exists $\Gamma_n=\{x_Q: Q\in \D_n\}$ with $x_Q \in Q$ such that any three distinct points of $\Gamma_n$ is not similar to $P$. Moreover there exists $\varepsilon=\varepsilon_n$ such that the  following two conditions hold.

$(C_1)$ $B(x_Q,\varepsilon) \subset Q $ for each $Q\in \D_n$.

$(C_2)$ For any $\{a_1, a_2, a_3\} \subset \bigcup_{Q\in \D_n} B(x_Q,\varepsilon)$ which is  similar to $P$, there exists $Q\in \D_n$ such that $\{a_1, a_2, a_3\} \subset  B(x_Q,\varepsilon)$. 
\end{lemma}
\begin{proof}
Let $\D_n=\{Q_i: 1\leq i\leq 2^{nd}\}$. Assume that we have chosen $m$ points $K_m=\{x_i : 1\leq  i\leq m\}$ with $x_i\in Q_i, 1\leq i\leq m$ such that any three distinct  points of $K_m$ is not similar to $P.$ 
For any two points $x_i,x_j$ of $K_m$, by Lemma \ref{lem:triangle} we obtain that the set 
\[
\{x\in \RR^{d}: P\backsim\{x_{i},x_j,x\}\}
\]
has Lebesgue measure zero. Note that there are at most $m(m-1)/2$ pairs of $(i,j)$. It follows that there exists an interior point $x_{m+1}$ of $Q_{m+1}$ such that any three points of 
$\{x_i: 1\leq i \leq m+1\}$ is not similar to $P.$ We use the same way to find points $x_{m+2},\cdots$. In the end, we obtain a point $x_{2^{dn}}$ from $Q_{2^{dn}}$. Let $\Gamma_n$ be the collection of chosen points. 

Since each $x_i$ is an interior point of $Q_i, 1\leq i\leq 2^{nd}$, there is $\varepsilon'>0$ such that the condition $C_1$ holds. Observe that for any three distinct points $\{x_{i_1},x_{i_2},x_{i_3}\} \subset \Gamma_n$, there is a positive constant $\varepsilon_{i_1,i_2,i_3}$ such that 
any there points $\{a_1,a_2,a_3\}$ with $a_{k}\in B(x_{i_k},\varepsilon_{i_1,i_2,i_3}), k=1,2,3$ is not similar to $P$. Let $\varepsilon''$ be the minimal value over all the possible $\varepsilon_{i_1,i_2,i_3}$, and 
$\varepsilon=\min\{\varepsilon',\varepsilon''\}$. Thus we complete the proof.
\end{proof}

\begin{definition}
Recall that the points $\{x_1,x_2,\cdots,x_m\}$ are called affinely independent if the vectors $x_2-x_1,\cdots, x_{m}-x_1$ are linearly independent.
Let $A\subset \RR^{d},d \geq 2$. We say $A$ is affinely independent if any $k ( 3\leq  k\leq d+1)$ distinct points of $A$ is affinely independent.
\end{definition}

Let $A\subset \RR^{d}, d\geq 2$ be a set that intersects any $(d-1)$-dimensional plane in at most $d$ points. Observe that this is equivalent to say that any $(d+1)$ points $\{x_1,\cdots,x_{d+1}\}\subset A$ is affinely independent.

\begin{lemma}\label{lem:H-set}
For each $\D_n, n\in \NN,$ there exists $\Gamma_n=\{x_Q: Q\in \D_n\}$ with $x_Q \in Q$ such that $\Gamma_n$ is affinely independent. Moreover there exists $\varepsilon=\varepsilon_n$ such that the  following two conditions hold.

$(C_1)$ $B(x_Q,\varepsilon) \subset Q $ for each $Q\in \D_n$.

$(C_2)$ For any $\{a_1,\cdots, a_{d+1}\} \subset \bigcup_{Q\in \D_n} B(x_Q,\varepsilon)$ which is not affinely independent, there exists $Q\in \D_n$ and $\{a_i, a_j\} \subset \Gamma_n$ such that $\{a_i, a_j\} \subset  B(x_Q,\varepsilon)$. 
\end{lemma}
\begin{proof}
Let $\D_n=\{Q_i: 1\leq i\leq 2^{nd}\}$. Assume that we have chosen $m$ points $K_m=\{x_i : 1\leq  i\leq m\}$ with $x_i\in Q_i, 1\leq i\leq m$ such that $K_m$ is affinely independent. 
Observe that for any $d$ points $x_{i_k}, 1\leq k \leq d$ of $K_m$, the set 
\[
\{x\in \RR^{d}: \{x\}\cup\{x_{i_k}:1 \leq k\leq  d\}   \text{ is not affinely independent }\}
\]
has Lebesgue measure zero. Note that there are at most finite elements of $(i_i,i_2,\cdots,i_d)$. It follows that there exists an interior point $x_{m+1}$ of $Q_{m+1}$ such that  
$\{x_i: 1\leq i \leq m+1\}$ is affinely independent. We use the same way to find points $x_{m+2},\cdots$. In the end, we obtain a point $x_{2^{dn}}$ from $Q_{2^{dn}}$. Let $\Gamma_n$ be the collection of chosen points.

Since each $x_i$ is an interior point of $Q_i, 1\leq i\leq 2^{nd}$, there is  $\varepsilon'>0$ such that the condition $C_1$ holds. Observe that for any $d+1$ distinct points $\{x_{i_1},\cdots,x_{i_{d+1}}\} \subset \Gamma_n$, there is a positive constant $\varepsilon_{i_1,\cdots,x_{i_{d+1}}}$ such that 
any $(d+1)$ points $\{a_{1},\cdots,a_{d+1}\}$ with 
\[
a_{k}\in B(x_{i_k},\varepsilon_{i_1,\cdots,x_{i_{d+1}}}), 1\leq k\leq d+1
\] is not affinely independent. Let $\varepsilon''$ be the minimal value over all the possible $\varepsilon_{i_1,\cdots,x_{i_{d+1}}}$and $\varepsilon=\min\{\varepsilon',\varepsilon''\}$.  
\end{proof}

In fact we can also choose the sets $\Gamma_n$ of Lemma \ref{lem:distribution} and Lemma \ref{lem:H-set} in a probability way. For each $Q\in \D_n$, we randomly choose a point $x_Q \in Q$ under the law of 	uniform distribution. The choices are independent for different cubes of $\D_n$. Denote by $\Gamma_n^{\omega}$ the random chosen points. It is not hard to show that with probability one $\Gamma_n^{\omega}$ has the same properties as $\Gamma_n$. We show the outline for this argument.

\begin{proposition} With probability one $\Gamma_n^{\omega}$ has the same properties 
as $\Gamma_n$ in Lemma \ref{lem:distribution} 
and Lemma \ref{lem:H-set}.
\end{proposition}
\begin{proof}
Let $\D_n=\{Q_1, Q_2,\cdots,Q_{2^{nd}}\}$. Lemma \ref{lem:triangle} implies that  conditional on $x_1\in Q_1,x_2\in Q_2$, the probability of the event
$P \backsim\{x_1,x_2,x_3\}$ is zero. Therefore we have that $\PP(P\backsim\{x_1,x_2,x_3\} )=0.$ It follows that 
\begin{align*}
\PP(\text{ exists } \{x_i,&x_j,x_k\} \subset \Gamma_n^{\omega} \text{ such that } \{x_i,x_j,x_k\}\backsim P )\\
&\leq \sum_{i,j,k}\PP(\{x_i,x_j,x_k\}\backsim P)=0.
\end{align*}
Thus we obtain that with probability one any three points of $\Gamma_n^{\omega}$ is not similar to $P$. 

Observe that 
\[
\PP(\{ x_1, x_2, x_3\} \text{ is affinely independent })=1.
\] 
Let $A_k$ be the event 
\[
\{x_1,x_2,\cdots x_k\} \text{ is affinely independent, } 3\leq k\leq d+1.
\]
Then it is not hard to see that $\PP(A_{k+1} \big | A_k)=1$, and $\PP(A_{k+1}|  A_k^{c})=0.$
Thus we have
\begin{align*}
\PP(A_{d+1})&=\PP(A_{d+1}| A_d)\PP(A_d)+\PP(A_{d+1}| A_d^{c})\PP(A_d^{c}) \\
&=\PP(A_d)=\cdots=\PP(A_3)=1.
\end{align*}
It follows that $\PP(A_{d+1}^{c})=0$, and thus
\begin{align*}
&\PP(  \Gamma_n^{\omega} \text{ is not affinely independent })\\
&\leq \sum_{i_{1},\cdots,i_{d+1}}\PP(\{x_{i_1},\cdots x_{i_{d+1}}\} \text{ is not affinely independent })=0.
\end{align*}

Since the boundary of cube has Lebesgue measure zero, we obtain that  with probability one $x_i$  is an interior point of $Q_i, 1\leq i\leq 2^{nd}$.
\end{proof}

\begin{proof}[Proof of Theorem \ref{thm:type}]
Let $ P=\{p_1,p_2,p_3\} \subset [0,1]^{d}$.  For each $n\in \NN$, let $\Gamma_n$ be the set in Lemma \ref{lem:distribution} and $P_n$ be the power set of $\Gamma_n$. Recall that the power set of a set $X$ is the collection of all the subset of $X$. Let 
\[
\G= \bigcap^{\infty}_{k=1}\bigcup^{\infty}_{n=k} \bigcup_{\g\in P_n}U_{d_H}(\g,\varepsilon_n),
\]
where $U_{d_H}(\g,\varepsilon_n)$ is an open set of $(\K, d_H)$ with center $\g$ and radius $\varepsilon_n$. Note that $\{\g: \g \in P_n, n\in \NN\}$ is a countable dense subset in $\K.$ Thus  
\[
\bigcup^{\infty}_{n=k} \bigcup_{\g\in P_n}U_{d_H}(\g,\varepsilon_n)
\]
is a dense open set in $\K$. It follows that the complementary set of $\G$ is of first category. In the following we intend to show that any element of $\G$ does not contain a similar copy of $\{p_1,p_2,p_3\}$. 

Let $E\in \G$, then there exist $n_k\nearrow \infty$ and $\g_{n_k} \in P_{n_k}$ such that $E\in \bigcap^{\infty}_{k=1} U_{d_H}(\g_{n_k},\varepsilon_{n_k})$.  Suppose that there is $\{x_1,x_2,x_3\} \subset E$ which is similar to $F$. By the condition $C_2$ of Lemma \ref{lem:distribution}, there is $Q\in \D_{n_k}$ such that 
\[
\{x_1,x_2,x_3\} \subset B(x_Q, \varepsilon_{n_k}),
\]
 and hence 
\[
 |\{x_1,x_2,x_3\}|\leq 2\varepsilon_{n_k}\leq \sqrt{d}2^{-n_k}\rightarrow 0. 
\]
This is a contradiction. Thus we complete the proof.
\end{proof}

\begin{proof}[Proof of Theorem \ref{thm:coline}]
For each $n\in \NN$, let $\Gamma_n$ be the set  in Lemma \ref{lem:H-set} and $P_n$ be the power set of $\Gamma_n$. Let 
\[
\G= \bigcap^{\infty}_{k=1}\bigcup^{\infty}_{n=k} \bigcup_{\g\in P_n}U_{d_H}(\g,\varepsilon_n).
\]
Then the complementary set of $\G$ is of first category.  

Let $E\in \G$, then there exist $n_k\nearrow \infty$ and $\g_{n_k} \in P_{n_k}$ such that $E\in \cap^{\infty}_{k=1} U_{d_H}(\g_{n_k},\varepsilon_{n_k})$. Suppose that $E$ is not affinely independent. Thus there exists $\{a_1,\cdots,a_{d+1}\} \subset E$ such that $\{a_1,\cdots,a_{d+1}\}$ is not affinely independent. For each $n_k$, there exists $\g\in \Gamma_n$ such that   
\[
\{a_1,\cdots,a_{d+1}\}\subset \bigcup_{x\in \gamma}B(x,\varepsilon_{n}).
\]
By the condition $C_2$ of Lemma \ref{lem:H-set}, we obtain that there exists two distinct points $a_i, a_j$ with $|a_i,a_j|\leq 2\varepsilon_n.$ Note that we may choose $\varepsilon_n$ such that $\varepsilon_n \searrow 0$. It follows that there exist two points of $\{a_1,\cdots,a_{d+1}\}$ with distance zero which is a contradiction. 
\end{proof}

\begin{remark}
Let $E\subset \RR^{d}, d\geq 2$. We say $E$ contains the angle $\theta$ if there are three points $\{x,y,z\}\subset E$ such that the angle between the vectors $y-x$ and $z-x$ is $\theta$, and write $\angle \theta \in E.$ For some results on this topic and further references we refer to \cite{Chan, Shmerkin2}. 

Let $\theta \in [0,\pi)$ and $a,b \in \RR^{d}, d \geq 2$. Then by some elementary geometric arguments, we have
\[
\mathcal{L}^{d}(\{x\in \RR^{d}: \angle \theta \in \{a,b,x\}\})=0.
\]
It follows that for each $n\in \NN$, there is  $\Gamma_n=\{x_Q: Q\in \D_n\}$ such that $\Gamma_n$ does not contain the angle $\theta$. 
Applying the similar argument in the proofs of Theorems \ref{thm:type} and \ref{thm:coline}, we obtain that a typical compact set of $\K([0,1]^{d}),d\geq 2$ does not contain the angle $\theta$. We omit the details here.
\end{remark}

\section{Proof of Theorem \ref{thm:hitting}}\label{section:hitting}

\begin{proof}[Proof of Theorem \ref{thm:hitting}.]
Suppose  $A$ is a nowhere dense subset of $[0,1]^d$. Let $E\in \K, \varepsilon=2^{-n}\sqrt{d}$. Assume first that $E\cap A\neq  \emptyset$. We define 
\[
\E_n=\{Q\in\D_n: Q\cap E \neq \emptyset\}=\E_n'\cup \E_n''
\]
where 
\[
\E_n'=\{Q\in \E_n: Q \cap A =\emptyset\},\, \E_n''=\E_n \backslash \E_n'.
\]
For every $Q \in \E_n'$, let $c_Q$ be the center point of $Q$. For every $Q\in \E_n'',$ since $A$ is nowhere dense, there exists $x_Q\in Q, r_Q >0$ such that 
\[
U(x_Q,r_Q) \subset Q \text{ and } U(x_Q,r_Q) \cap A \neq \emptyset.
\]
Let $F$ be the collection of points $c_Q$ for $Q\in \E_n'$ and $x_Q$ for $Q\in \E_n''$. Then $F\in U_{d_H}(E, 2^{-n}\sqrt{d})$. Let 
\[
\varepsilon'=\min\{r_Q: Q\in \E_n''\}.
\]
Then $U_{d_H}(F,\varepsilon') \cap \K_A=\emptyset$. 

For the case $E\cap A =\emptyset$ we have that  $\E''=\emptyset.$ Let $F$ be the collection of points $c_Q$ for $Q\in \E_n'$. Then 
\[
F\in U_{d_H}(E, 2^{-n}\sqrt{d}) \text{ and } U_{d_H}(F, 2^{-n-1}) \cap \K_A=\emptyset.
\]
By the arbitrary choice of $E\in \K$ and $\varepsilon=2^{-n}\sqrt{d}$, we obtain that $\K_A$ is nowhere dense in $\K$.

Now we assume that there is an open ball $U \subset \overline{A}$ where $\overline{A}$ is the closure of $A$. Let $\K(U)$ be all the compact subsets of $U$, then $\K(U)$ is an open set in $\K$. Observe that  $\K(U)\subset \overline{\K_A}$. Thus we obtain that if $\K_A$ is nowhere dense in $\K$ then $A$ is nowhere dense in $[0,1]^{d}$. 
\end{proof}

\section{Proof of Theorem \ref{thm:arithmetic}}\label{section:arithmetic}

For $A,B \subset \RR$ we define their sum set 
\[
A+B=\{a+b: a\in A, b\in B\}.
\] 
Let $\lambda \in \RR$ and $\lambda A=\{\lambda\times a: a \in A\}$. For $m\in \NN$, define
\[
S^{m}(A):=\Big\{\sum^{m}_{i=1}x_i: x_i \in A, 1\leq i\leq m \Big\},
\] 
\[
A^{m}=\{(x_1,\cdots,x_m): x_i \in A, 1 \leq i \leq m\},
\]
and
\[
T^{m}(A):=\Big\{x_1\times \cdots \times x_m: x_i \in A, 1\leq i\leq m \Big\}.
\]
Let $P(x)=\sum^{n}_{k=0}a_kx^{k}, a_k \in \RR$ be a polynomial. For a set $A \subset \RR$, let
\[ 
\widetilde{P}(A):=\sum_{k=0}^{n}a_k T^{k}(A).
\]
Note that $\widetilde{P}(A)$ is the sum set of $\{a_k T^{k}(A)\}^{n}_{k=0}$. The Hausdorff dimension of $E$ is defined as 
\[
\dim_H E=\inf \{s \geq 0: \mathcal{H}^s(E)=0\},
\]
where $\mathcal{H}^s(E)=\lim_{\delta \rightarrow 0}\mathcal{H}^s_\delta(E),$
and 
\[
\mathcal{H}^s_\delta(E)= \inf \Big\{\sum^{\infty}_{i=1} |U_i|^s: E\subset \bigcup_{i\in \NN} U_i, |U_i| \leq \delta, i \in \NN  \Big\}.
\]

For each $n\in \NN$, let $\varepsilon_n=2^{-n^{2}}$,
\[
\D_n'=\{\frac{k}{2^{n}}: 0\leq k\leq 2^{n}\},
\]
and $P_n$ be the power set of $\D_n'$. Define 
\[
\G= \bigcap_{k=1}^{\infty} \bigcup^{\infty}_{n=k}\bigcup_{\g \in P_n}U_{d_H}(\g, \varepsilon_n).
\]
Applying the same argument as in the proof of Theorem \ref{thm:type}, we have that the complementary of $\G$ is of first category in $\K$.

\begin{lemma}\label{lem:product}
Let $A\in \G$ then $\dim_H A^{m}=0$ for any $m\in \NN$.
\end{lemma}
\begin{proof}
For any $k\in \NN$ there exist $n\geq k$ and 
\[
\g=\{x_1,\cdots,x_N\}\in P_n
\]
such that $A \in U_{d_H}(\g,\varepsilon_n)$. It follows that 
\[
A\subset \bigcup_{i=1}^{N}B_i
\]
where $B_i:=B(x_i,\varepsilon_n), 1\leq i\leq N.$ Let $m\in \NN$, then
\[
A^{m}\subset \bigcup_{i_1\cdots i_m \in \I^{m}} B_{i_1}\times \cdots \times B_{i_m}
\]
where $\I=\{1,2,\cdots,N\}$. Note that 
\[
|B_{i_1}\times \cdots \times B_{i_m}| \leq \sqrt{m}\varepsilon_n \text{ for any } i_1\cdots i_m \in \I^{m}.
\]
Since $N\leq 2^{n}+1$ and $\varepsilon_n=2^{-n^{2}}$, for any $s>0$ we have 
\[
\mathcal{H}^{s}_{\varepsilon_n\sqrt{m}}(A^{m})\leq N^{m}(\varepsilon_n \sqrt{m})^{s} \leq 2^{n+1}2^{-n^{2}s}\sqrt{m}^{s}.
\]
It follows that $\mathcal{H}^{s}(A^{m})=0$. By the arbitrary choice of $s>0$ we have that $\dim_H A^{m}=0$. Thus we complete the proof.
\end{proof}

\begin{proof}[Proof of Theorem \ref{thm:arithmetic}]
It is clear that $S^{m}(A)=\sqrt{m}\pi_{e}(A^{m})$ where 
$\pi_{e}(A^{m})$ is the orthogonal projection of $A^{m}$ on to the line with direction $e=\sqrt{m}^{-1}(1, 1, \cdots,1).$ Thus $\dim_H S^{m}(A)\leq \dim_H A^{m}$. Therefore by Lemma \ref{lem:product} we obtain $\dim_H S^{m}(A)=0.$  

Suppose that 
\[
P(x)=\sum^{n}_{k=0}a_kx^{k}, a_k \in \RR, a_n\neq 0.
\] 
Does not lose general we may assume $ a_0=0$. Note that  
\[
\widetilde{P}(A)=\{\sum_{k=1}^{n}a_kx_{k,1}\cdots x_{k,k}: 1\leq i \leq k, x_{k,i} \in A\}.
\] 
Define a new function 
\[
\varphi: [0,1]^{\frac{(n+1)n}{2}}\longrightarrow \RR
\]
by 
\[
\varphi(x_{1,1}, x_{2,1},x_{2,2},\cdots, x_{nn})=\sum_{k=1}^{n}a_kx_{k,1}\cdots x_{k,k}.
\]
By the mean value theorem we have that $\varphi$ is a Lipschitz map on $[0,1]^{\frac{(n+1)n}{2}}$. Observe that 
\[
\widetilde{P}(A)=\varphi (A^{\frac{(n+1)n}{2}}).
\]
Thus by Lemma \ref{lem:product} and the fact that Lipschitz map will not increase the Hausdorff dimension, we obtain that $\dim_H \widetilde{P}(A)=0.$
\end{proof}

\begin{remark}
Let $A\subset \RR^{d}$. Then we can also consider the sets $S^{m}(A), A^{m}$. By applying the same arguments in Lemma \ref{lem:product} and in the proof of Theorem \ref{thm:arithmetic}, we have that for a typical compact set $A\in \K([0,1]^{d})$,
\[
\dim_H A^{m}=0 \text{ and } \dim_{H} S^{m}(A)=0 \text{ for any }m\in \NN.
\]
We omit the details here.
\end{remark}

Denote
\[
e^{A}:=\sum_{n=0}^{\infty} \frac{T^{n}(A)}{n!}.
\] 
We consider $e^{A}$ as the limit point of $S_m$ in the space $(\K(\RR), d_H)$ where 
\[
S_m:=\sum_{n=0}^{m} \frac{T^{n}(A)}{n!}
\]  
is a sum set of $\{\frac{T^{n}(A)}{n!}\}^{m}_{n=0}$. Note that $\{S_m\}_{m\in \NN}$ is a Cauchy sequence in the complete metric space $(\K(\RR), d_H)$. Thus the set $e^{A}$ is well defined.

\begin{question}  
Is it true that a typical $A \in \K$ has $\dim_H e^{A}=0$?
\end{question}

\section{Typical continuous functions}

Let $\C=\C([0,1])$ be all the continuous functions on $[0,1]$. The distance of continuous functions $f,g\in \C$ is defined by 
\[
d(f,g)=\max\{|f(x)-g(x)|: x \in [0,1]\}.
\]
Let $A \subset [0,1]\times \RR$. Define
\[
\C_A=\{f\in \C: G(f) \cap A \neq \emptyset\}
\]
where $G(f)$ is the graph of function $f$. Let $x\in [0,1]$ and $V_x=\{x\}\times \RR$. In the following we only consider the case $A\subset V_x$. For this special case, we have the following  similar result to Theorem \ref{thm:hitting}.

\begin{proposition}\label{pro:continuous}
A subset $A\subset V_x, x \in [0,1]$ is nowhere dense  in $V_x$ if and only if $\C_A$ is nowhere dense in $\C$. 
\end{proposition}
\begin{proof}
Suppose $A$ is nowhere dense in $V_x$. Let $U_{\C}(f,\varepsilon)$ be an open ball in $\C$ with center $f$ and radius $\varepsilon$. Then by the nowhere dense of $A$ there exist $g\in \C, \varepsilon'>0$ such that 
\[
U(g(x), \varepsilon')\cap A=\emptyset, \text{ and } U(g(x), \varepsilon') \subset U(f(x),\varepsilon).
\]
Here $U(f(x), \varepsilon)$ is an open ball in $V_x$ with center $f(x)$ and radius $\varepsilon$. Note that $U_{\C}(g,\varepsilon') \cap \C_A=\emptyset$. By the arbitrary choice of $f\in \C$ and $\varepsilon$ we obtain that $\C_A$ is nowhere dense.

By applying the same argument as in the proof of Theorem \ref{thm:hitting}, we obtain that if $\C_A$ is nowhere dense then $A$ is nowhere dense. 
\end{proof}



Applying the same argument as in the introduction, we obtain that if  $A\subset V_x, x \in [0,1]$ is of first category in $V_x$ then $\C_A$ is of first category in $\C$. Again we do not know that if the converse claim is also true. Let $z \in [0,1] \times \RR$ then Proposition \ref{pro:continuous} claims that $\C_{z}$ is nowhere dense in $\C$. Since the rational points in plane is countable, we obtain that the graph of a typical continuous function of $\C$ does not contain any rational points in plane.

Maga \cite{Maga} proved that for any distinct points $\{x,y,z\}\subset \RR^{2}$, there exists a compact set $ E \subset \RR^{2}$  with $\dim_H E=2$ and $E$ does not contain a similar copy of  $\{x,y,z\}$. Motived by this result and Theorem \ref{thm:type} we ask the following question.

\begin{question}
Does the graph of a continuous function with Hausdorff dimension larger than one contains three points which are the vertices of an equilateral triangle?
\end{question}

\noindent\textbf{Acknowledgements.} I thank Ville Suomala for helpful discussions, and Meng Wu and Wen Wu for valuable comments.

\end{document}